\documentclass[12pt]{amsart}
\usepackage{amsmath}
\usepackage{mathrsfs}
\usepackage{graphicx}
\def\e{\epsilon}

\def\p{\partial}

\newcommand\C{{\mathbb C}}

\def\be{\begin{equation}}
\def\ee{\end{equation}}
\def\bee{\begin{equation*}}
\def\eee{\end{equation*}}

\def\be{{\beta}}

\def\e{\epsilon}

\def\p{\partial}

\def\C{\Bbb C}

\def\p{\partial}

\def\p{\partial}

\def\C{\Bbb C}

\def\e{\epsilon}

\def\p{\partial}

\def\C{\Bbb C}

\def\p{\partial}

\def\p{\partial}

\def\C{\Bbb C}

\def\la{\langle}
\def\ra{\rangle}

\def\vv<#1,#2>{\left\la#1,#2\right\ra}

\newtheorem{thm}{Theorem}[section]

\newtheorem{lem}{Lemma}[section]

\theoremstyle{definition}

\theoremstyle{remark}

\numberwithin{equation}{section}

\def\be{\begin{equation}}
\def\ee{\end{equation}}
\def\ord{{\rm ord}}

\begin{document}
\title{Rigidity of the sharp Bezout estimate on nonnegatively curved Riemann surfaces}

\author{Chengjie Yu$^1$}
\address{Department of Mathematics, Shantou University, Shantou, Guangdong, 515063, China}
\email{cjyu@stu.edu.cn}
\author{Chuangyuan Zhang}
\address{Department of Mathematics, Shantou University, Shantou, Guangdong, 515063, China}
\email{12cyzhang@stu.edu.cn}
\thanks{$^1$Research partially supported by GDNSF with contract no. 2021A1515010264 and NNSF of China with contract no. 11571215.}
\renewcommand{\subjclassname}{%
  \textup{2010} Mathematics Subject Classification}
\date{}
\begin{abstract}
In this short note, by using a general three circle theorem, we show the rigidity of the sharp Bezout estimate first found by Gang Liu on nonnegatively curved Riemann surfaces.
\end{abstract}
\maketitle\markboth{Yu \& Zhang}{Rigidity of a Bezout estimate}
\maketitle
\section{Introduction}
On a complete noncompact K\"ahler manifold $(M,g)$, the space of holomorphic functions on $M$ is denoted as $\mathcal O(M)$. A function $f$ on $M$ is said to be of polynomial growth if there are some positive constants $C$ and $d$ such that \begin{equation}\label{eq-polynomial-growth}
|f(x)|\leq C(1+r^d(x))\ \forall x\in M,
\end{equation}
where $r(x)$ the distance function to some fixed point. The space of holomorphic functions of polynomial growth is denoted as $\mathcal P(M)$, and the space of holomorphic functions on $M$ satisfying \eqref{eq-polynomial-growth} for some $C>0$ is denoted as $\mathcal O_d(M)$. For any nonzero $f\in \mathcal P(M)$, the degree of $f$ is defined as
\begin{equation}
\deg(f)=\inf\{d>0\ |\ f\in \mathcal O_d(M)\}.
\end{equation}
For $x\in M$ and $f\in \mathcal O(M)$, $\ord_x(f)$ means the vanishing order of $f$ at $x$.

The purpose of this short note is to give a proof of the following result.
\begin{thm}\label{thm-Bezout}
Let $(M,g)$ be a noncompact Riemann surface equipped with a complete conformal metric $g$  such that the Guassian curvature of $g$ is nonnegative. Then, for any nonzero holomorphic function $f$ of polynomial growth on $M$,
\begin{equation}\label{eq-Bezout}
\sum_{x\in M}\ord_x(f)\leq \deg(f).
\end{equation}
Moreover, the equality of \eqref{eq-Bezout} holds for some nonconstant holomorphic function $f$ of polynomial growth if and only if $(M,g)$ is biholomorphically isometric to $\C$ equipped with the standard metric.
\end{thm}
The estimate \eqref{eq-Bezout} was first obtained in the unpublished preprint \cite{Liu1} by Liu. It is clear that the estimate \eqref{eq-Bezout} is sharp because the equality of  the estimate holds for any nonzero polynomial on $\C$. In this short note, we give an alternative proof of \eqref{eq-Bezout} using the following general three circle theorem on Riemann surfaces, in a similar spirit of \cite{Liu2} and characterize the rigidity of the estimate \eqref{eq-Bezout}.

\begin{thm}\label{thm-three-circle}Let $M$ be a Riemann surface and $\Omega$ be an open subset of $M$. Let $u$ and $v$ be two continuous functions on $\Omega$ that are subharmonic  and  supharmonic  respectively. Suppose that $v:\Omega\to (\inf_{\Omega}v, \sup_\Omega v)$ is proper and $M_v(u,t)=\max_{x\in S_v(t)}u(x)$ is an increasing function for $t\in(\inf_{\Omega}v, \sup_\Omega v)$ where $S_v(t)=\{x\in\Omega\ |\ v(x)=t\}$. Then, $M_v(u,t)$ is a convex function of $t\in (\inf_\Omega v,\sup_\Omega v)$.
\end{thm}
It seems that the original proof in \cite{Liu1} can not characterize the rigidity of the estimate \eqref{eq-Bezout}. In higher dimension, there is an estimate in similar spirit which is not sharp in \cite[Proposition 7.2]{CTZ}.

In algebraic geometry, the Bezout theorem relates the number of roots and the degrees of polynomials.  So, an estimate such as \eqref{eq-Bezout} is called a Bezout estimate (See \cite[P. 244]{Mok}). The Bezout estimate \eqref{eq-Bezout} is stronger than the sharp vanishing order estimate:
\begin{equation}\label{eq-vanishing-order}
\ord_x(f)\leq \deg(f)
\end{equation}
for any nonzero $f\in \mathcal P(M)$ and any $x\in M$. Vanishing order estimate that is not sharp was first obtained by Mok \cite{Mok} on complete noncompact K\"ahler manifolds with certain geometric conditions. In \cite{Ni}, Ni obtained the sharp estimate \eqref{eq-vanishing-order} under the assumptions of nonnegative holomorphic bisectional curvature and maximal volume growth. The second assumption of Ni's result was later removed by Chen-Fu-Yin-Zhu \cite{CFYZ} and finally the first assumption of Ni' result was relaxed to nonnegative holomophic sectional curvature by Liu \cite{Liu2}. According to all these works, it seems that there must be a  higher dimensional analogue of the sharp Bezout estimate \eqref{eq-Bezout}.
\section{Proof of the theorem}
Although the proof of the general three circle theorem (Theorem \ref{thm-three-circle}) is almost the same as that of the classical Hadamard's three circle theorem on $\C$, we give a proof to it for completeness.
\begin{proof}[Proof of Theorem \ref{thm-three-circle}]
For any $\inf_\Omega v<t_1<t_2<t_3<\sup_\Omega v$, let 
\begin{equation}
\tilde v=\frac{M_v(u,t_3)-M_v(u,t_1)}{t_3-t_1}v+\frac{t_3M_v(u,t_1)-t_1M_v(u,t_3)}{t_3-t_1}.
\end{equation}
Because $M_v(u,t_3)-M_v(u,t_1)\geq 0$, $\tilde v$ is supharmonic. Moreover, it is clear that 
$$\tilde v\geq u$$
on $S_v(t_1)\cup S_v(t_3)$. Hence, by maximum principle, 
$$\tilde v\geq u$$
in $A_v(t_1,t_3):=\{x\in \Omega\ |\ t_1<v(x)<t_3\}.$ Thus, for any $x\in S_v(t_2)$, 
\begin{equation}
\begin{split}
u(x)\leq& \tilde v(x)\\
=&\frac{M_v(u,t_3)-M_v(u,t_1)}{t_3-t_1}t_2+\frac{t_3M_v(u,t_1)-t_1M_v(u,t_3)}{t_3-t_1}\\
=&\frac{t_2-t_1}{t_3-t_1}M_v(u,t_3)+\frac{t_3-t_2}{t_3-t_1}M_v(u,t_1).\\
\end{split}
\end{equation}
Hence
\begin{equation}
M_v(u,t_2)\leq \frac{t_2-t_1}{t_3-t_1}M_v(u,t_3)+\frac{t_3-t_2}{t_3-t_1}M_v(u,t_1).
\end{equation}
This completes the proof of the theorem.
\end{proof}
Before the proof of Theorem \ref{thm-Bezout}, we need the following two lemmas.
\begin{lem}\label{lem-limit-r-0}
Let $(M,g)$ be a noncompact Riemann surface equipped a conformal metric $g$ and $f$ be a nonzero holomorphic function on $M$. Let $p_1,p_2,\cdots,p_n$ be $n$ distinct roots of $f$, $m_i=\ord_{p_i}(f)$ and $\rho=\prod_{i=1}^nr_{p_i}^{m_i}$. Then,
\begin{equation}
\lim_{r\to {0^+}}\frac{\log M_\rho(|f|,r)}{\log r}=1.
\end{equation}
Here $r_p$ means the distance function to $p$.
\end{lem}
\begin{proof}
Let $r_0=\frac{1}{4}\min\{r(p_i,p_j)\ |\ 1\leq i<j\leq n \}$ and $\delta_0\in (0,r_0]$ be such that for any $i=1,2,\cdots,n$ and $x\in B_{p_i}(\delta_0)$,
\begin{equation}\label{eq-order-f}
 c_1 r_{p_i}^{m_i}(x)\leq|f(x)|\leq C_2 r_{p_i}^{m_i}(x)
\end{equation}
for some positive constants $c_1$ and $C_2$. Let $\rho_0=\delta_0^N$ where $N=\sum_{i=1}^nm_i$. Then, for any $x\in M$ such that $\rho(x)=r<\rho_0$, $x\in B_{p_i}(\delta_0)$ for some $i=1,2,\cdots,n$.  Then, for any $j\neq i$, by the triangle inequality,
\begin{equation}
\delta_0<r(p_i,p_j)-r_{p_i}(x)\leq r_{p_j}(x)\leq r(p_i,p_j)+r_{p_i}(x)\leq R_0+\delta_0
\end{equation}
where $R_0=\max\{r(p_i,p_j)\ |\ 1\leq i<j\leq n\}.$ Thus,
\begin{equation}
\frac{r}{(R_0+\delta_0)^{N-m_i}}\leq r_{p_i}^{m_i}(x)=\frac{\rho(x)}{r_{p_1}^{m_1}(x)\cdots \widehat{r_{p_i}^{m_i}(x)}\cdots r_{p_n}^{m_n}(x)}\leq \frac{r}{\delta_0^{N-m_i}}.
\end{equation}
Combining this with \eqref{eq-order-f}, we know that
\begin{equation}
\frac{c_1r}{(R_0+\delta_0)^{N-m_i}}\leq |f(x)|\leq \frac{C_2r}{\delta_0^{N-m_i}}.
\end{equation}
So, for any $x\in M$ with $\rho(x)=r<\rho_0$, we have
\begin{equation}
c_3 r\leq |f(x)|\leq C_4r
\end{equation}
where $$c_3=\min\left\{\frac{c_1}{(R_0+\delta_0)^{N-m_i}}\ \Bigg|\ i=1,2,\cdots,n\right\}$$ and $$C_4=\max\left\{\frac{C_2}{\delta_0^{N-m_i}}\ \Bigg|\ i=1,2,\cdots,n\right\}.$$
Hence
\begin{equation}
\frac{C_4}{\log r}+1\leq \frac{\log M_\rho(|f|,r)}{\log r}\leq 1+\frac{c_3}{\log r}
\end{equation}
for any $r>0$ small enough. Then, by the squeezing rule, we complete the proof of the lemma.
\end{proof}
\begin{lem}\label{lem-compare-M}
Let $(M,g)$ be a noncompact Riemann surface equipped with a complete conformal metric $g$. Let $p_1, p_2,\cdots, p_n$ be $n$ distinct points on $M$, $m_1,m_2,\cdots,m_n\in \mathbb N$ and  $\rho=r_{p_1}^{m_1} r_{p_2}^{m_2}\cdots r_{p_n}^{m_n}$. Let $o$ be a fixed point on $M$,  $$R_0=\max\{r(o,p_i)\ |\ 1\leq i\leq n\}.$$ Then, for any holomorphic function $f$ on $M$ and any $r>(2R_0)^N$,
\begin{equation}
M_o(|f|, r^\frac1N-R_0)\leq M_\rho(|f|,r)\leq M_o(|f|,r^\frac1N+R_0)
\end{equation}
where $N=\sum_{i=1}^nm_i$ and $M_o(|f|,r)=\max_{x\in S_o(r)}|f(x)|$ with $S_o(r)=\p B_o(r)$.
\end{lem}
\begin{proof}
By the triangle inequality
\begin{equation}\label{eq-triangle}
r_o(x)-R_0\leq r_{p_i}(x)\leq r_o(x)+R_0,\  \forall x\in M,
\end{equation}
for $i=1,2,\cdots,n$. Then, for any $x$ with $\rho(x)=r>(2R_0)^N$, we have
\begin{equation}
(r_o(x)+R_0)^N\geq \rho(x)=r>(2R_0)^N
\end{equation}
which implies that
\begin{equation}
r_o(x)>R_0.
\end{equation}
Then, by the left hand side of \eqref{eq-triangle},
\begin{equation}
(r_o(x)-R_0)^N\leq \rho(x)=r.
\end{equation}
Thus
\begin{equation}
r_o(x)\leq r^\frac{1}{N}+R_0
\end{equation}
which means that $S_\rho(r)\subset B_o(r^\frac1N+R_0)$ for any $r>(2R_0)^N$. Therefore, by the principle of maximal modulus for holomorphic functions,
\begin{equation}
M_\rho(|f|,r)\leq M_o(|f|,r^\frac1N+R_0).
\end{equation}
On the other hand, for any $x\in S_o(r)$, by the right hand side of \eqref{eq-triangle},
\begin{equation}
\rho(x)\leq (r+R_0)^N.
\end{equation}
So, $S_o(r)\subset B_\rho((r+R_0)^N):=\{x\in M\ |\ \rho(x)\leq (r+R_0)^N\}$ and hence
\begin{equation}
M_o(|f|,r)\leq M_\rho(|f|,(r+R_0)^N)
\end{equation}
by the principle of maximal modulus for holomorphic functions again.
Thus
\begin{equation}
M_\rho(|f|,r)\geq M_o(|f|,r^\frac1N-R_0)
\end{equation}
for any $r>R_0^N$. This completes the proof of the theorem.
\end{proof}
We are now ready to prove Theorem \ref{thm-Bezout}.
\begin{proof}[Proof of Theorem \ref{thm-Bezout}]
Let $p_1,p_2,\cdots,p_n$ be $n$ distinct roots of $f$ with $m_i=\ord_{p_i}(f)$ for $i=1,2,\cdots,n$. Let $\rho=r_{p_1}^{m_1}r_{p_2}^{m_2}\cdots r_{p_n}^{m_n}$. Then, by the Laplacian comparison, we know that
\begin{equation}\label{eq-Laplacian-comparison}
\Delta r_{p_i}\leq \frac{1}{r_{p_i}}
\end{equation}
for $i=1,2,\cdots,n$ in the sense of distribution. So,
\begin{equation}
\Delta \log r_{p_i}\leq 0
\end{equation}
on $M\setminus\{p_i\}$ in the sense of distribution which means that $\log r_{p_i}$ is a supharmonic function on $M\setminus\{p_i\}$ for $i=1,2,\cdots,n$. So $\log \rho=\sum_{i=1}^nm_i\log r_{p_i}$ is a supharmonic function on $\Omega=M\setminus\{p_1,p_2,\cdots,p_n\}$. Moreover, it is clear that $\log(|f|^2+\e)$ is a subharmonic function on $M$. Then, by Theorem \ref{thm-three-circle}, $\log M_\rho(|f|^2+\e,r)$ is a convex function of $\log r$. Letting $\e \to 0^+$, we know that $\log M_\rho(|f|,r)$ is a convex function of $\log r$.  Then, $\log M_\rho(|f|,r)-\log r$ is also a convex function of $\log r$. Therefore, for any $0<r_1<r_2<r_3$,
\begin{equation}\label{eq-convex}
\begin{split}
&\log M_\rho(|f|,r_2)-\log r_2\\
\leq& \frac{\log r_3-\log r_2}{\log r_3-\log r_1}(\log M_\rho(|f|,r_1)-\log r_1)+\frac{\log r_2-\log r_1}{\log r_3-\log r_1}(\log M_\rho(|f|,r_3)-\log r_3).
\end{split}
\end{equation}
Letting $r_1\to 0^+$ in the last inequality, and using Lemma \ref{lem-limit-r-0}. we get
 \begin{equation}\label{eq-increasing}
\begin{split}
\log M_\rho(|f|,r_2)-\log r_2\leq \log M_\rho(|f|,r_3)-\log r_3
\end{split}
\end{equation}
for any $0<r_2<r_3$.  Thus, letting $r_2=1$ in \eqref{eq-increasing}, we have
\begin{equation}\label{eq-f-linear}
M_\rho(|f|,r)\geq M_\rho(|f|,1)r
\end{equation}
for any $r>1$. Let $o\in M$ be a fixed point and $R_0=\max\{r(o,p_i)\ |\ 1\leq i\leq n\}$. Then, by Lemma \ref{lem-compare-M} and \eqref{eq-f-linear}, for any $r$ large enough,
\begin{equation}
M_o(|f|,r^\frac{1}{N}+R_0)\geq M_\rho(|f|,r)\geq M_\rho(|f|,1)r.
\end{equation}
Hence, for any $r$ large enough,
\begin{equation}
M_o(|f|,r)\geq M_\rho(|f|,1)(r-R_0)^N.
\end{equation}
So
\begin{equation}
\deg(f)\geq N=\sum_{i=1}^n\ord_{p_i}(f).
\end{equation}
This completes the proof of the Bezout estimate \eqref{eq-Bezout}.

We next come to prove the rigidity of \eqref{eq-Bezout}. Let $f$ be a nonconstant holomorphic function of polynomial growth such that equality of \eqref{eq-Bezout} holds. Let $d=\deg(f)$. It is then clear from Cheng's Liouville theorem for harmonic functions that $d\geq 1$. Let $p_1,p_2,\cdots,p_n$ be all the distinct roots of $f$, and $m_i=\ord_{p_i}(f)$ for $i=1,2,\cdots,n$. Then $d=\sum_{i=1}^n m_i$. By Lemma \ref{lem-compare-M}, we know that
\begin{equation}
\frac{\log M_o(|f|, r^\frac1d-R_0)}{\log r}\leq \frac{\log M_\rho(|f|,r)}{\log r}\leq \frac{\log M_o(|f|,r^\frac1d+R_0)}{\log r}
\end{equation}
when $r$ is large enough. Moreover, note that
\begin{equation}
\lim_{r\to +\infty}\frac{\log M_o(|f|,r^\frac1d+R_0)}{\log r}=\lim_{t\to+\infty}\frac{\log M_o(|f|,t)}{d\log(t-R_0)}=1
\end{equation}
and
\begin{equation}
\lim_{r\to+\infty}\frac{\log M_o(|f|, r^\frac1d-R_0)}{\log r}=\lim_{t\to+\infty}\frac{\log M_o(|f|,t)}{d\log(t+R_0)}=1
\end{equation}
where we have used the following conclusion proved in \cite{Liu2}:
\begin{equation}
\deg(f)=\lim_{r\to+\infty}\frac{\log M_o(|f|,r)}{\log r}
\end{equation}
for any nonzero $f\in \mathcal P(M)$. Thus, by the squeezing rule,
\begin{equation}
\lim_{r\to+\infty}\frac{\log M_\rho(|f|,r)}{\log r}=1.
\end{equation}
Then, by letting $r_3\to+\infty$ in \eqref{eq-convex}, we get
  \begin{equation}
\begin{split}
\log M_\rho(|f|,r_2)-\log r_2\leq \log M_\rho(|f|,r_1)-\log r_1
\end{split}
\end{equation}
for any $0<r_1<r_2$. Combining this with \eqref{eq-increasing}, we know that $$\log M_\rho(|f|,r)-\log r=C$$ for some constant $C$. For each $r>0$, let $z_r\in S_\rho(r)$ achieve the maximal modulus of $f$ on $S_\rho(r)$ and
\begin{equation}
F(x)=\log |f(x)|-\log \rho(x)-C.
\end{equation}
Then $F\leq 0$ and $F(z_r)=0$. So, $F$ as a subharmonic function on $\Omega$ achieves its maximum value on $z_r$. By strong maximum principle for subharmonic functions, we know that $F\equiv 0$ and hence
\begin{equation}
\log \rho(x)=\log |f(x)|-C
\end{equation}
is harmonic $\Omega$ since $\Delta\log|f|=0$ on $\Omega$. So
\begin{equation}
\Delta \log r_{p_1}=\Delta \log r_{p_2}=\cdots=\Delta \log r_{p_n}=0
\end{equation}
and the equality of the Laplacian comparison \eqref{eq-Laplacian-comparison} holds. Thus $(M,g)$ is flat and hence is biholomorphically isometric to $\C$ or a cylinder. Note that a cylinder admits no nonconstant holomorphic function of polynomial growth. So, $(M,g)$ is biholomorphically isometric to $\C$.
\end{proof}

\end{document}